\documentclass{amsart}

\usepackage{amssymb,latexsym}     
\usepackage[utf8]{inputenc}
\usepackage[english]{babel}

\usepackage{mathrsfs}
\usepackage{amsmath,amsthm}
\usepackage{graphicx}
\usepackage{tikz-cd}
\usepackage{adjustbox}
\usepackage{tikz}
\usetikzlibrary{matrix}
\usepackage{caption}

\newtheorem{Lemma}{Lemma}[section]

\newtheorem{Theorem}[Lemma]{Theorem}
\newtheorem{Conjecture}[Lemma]{Conjecture}
\newtheorem{Corollary}[Lemma]{Corollary}
\newtheorem{Statement}[Lemma]{Statement}
\theoremstyle{definition}
\newtheorem{Remark}[Lemma]{Remark}
\newtheorem{Definition}[Lemma]{Definition}
\newtheorem{Notation}[Lemma]{Notation}

\usepackage{float}
\usepackage{subcaption}

\usepackage[lmargin=0.8in,rmargin=0.8in,bmargin=0.5in,tmargin=0.8in]{geometry}
\begin{document}
\vspace*{-1cm}

\title{New explicit solution to the $N$-Queens Problem\\ and\\ its relation to the Millennium Problem}

\author{Dmitrii Mikhailovskii}
\address{School 564, Saint Petersburg, 190005 Russia} 
\email{mikhaylovskiy.dmitriy@gmail.com}

\begin{abstract}
Using modular arithmetic of the ring $\mathbb{Z}_{n+1}$ we obtain a new short solution to the problem of existence of at least one solution to the $N$-Queens problem on an $N \times N$ chessboard. It was proved, that these solutions can be represented as the Queen function with the width fewer or equal to $3$. It is shown, that this estimate could not be reduced. A necessary and sufficient condition of being a composition of solutions a solution is found. Based on the obtained results we formulate a conjecture about the width of the representation of arbitrary solution. If this conjecture is valid, it entails solvability of the $N$-Queens completion in polynomial time. The connection between the $N$-Queens completion and the Millennium $P$ vs $NP$ Problem is found by the group of mathematicians from Scotland in August $2017$.
\end{abstract}
\maketitle

\section*{Introduction}

The Millennium Problems are seven problems in mathematics that were stated by the Clay Mathematics Institute in $2000$. A correct solution to any of these problems results in a US $\$ 1000000$ prize being awarded by the institute to the discoverer. Currently, the only Problem that has been solved is so-called Poincare conjecture. Another of these $7$ Problems is related to the complexity of algorithms. Among these algorithms, polynomial algorithms are highlighted. The class of these algorithms is designated by $P$. Another class of algorithms are the algorithms which are able to check in polynomial number of steps that an answer is indeed a solution to a problem. Class of these algorithms is designated by $NP$. The Millennium Problem is the $P$ versus $NP$ problem. In August $2017$ a group of mathematicians from Scotland proved that $N$-Queens Completion Problem is $NP$-complete. Namely, if this problem can be solved in polynomial time, then $P$ is equal to $NP$.

Our work is devoted to the $N$-Queens Problem i.e. the well-known problem of placing $N$ chess queens on an $N\times N$ chessboard so that no two queens attack each other.

C.F. Gauss found $72$ solutions for $N=8$. But $24$ years later J.W.L. Glaisher proved using a method of determinants that for $N=8$ there are exactly $92$ solutions. The existence of a solution for arbitrary $N$ was proved by different authors using different methods. It was first proved by E. Pauls in $1874$. Now the number of different solutions $Q(N)$ is computed only for $4\leqslant N\leqslant 27$. Calculation of $Q(N)$ is related to the $N$-Queens Completion Problem (if we have $m<N$ queens on the board, is it possible to complete this board to the solution of the $N$-Queens Problem?). Existing methods of solving the completion problem stop working for $N\geqslant 1000$.

In this paper, we try to find such an algorithm.

We introduce the new way of representing the arrangements of queens as a {\bf Queen functions with width $k$} which is a map $[1,N] \to [1,N]$ which is defined on a partition of a segment $[1,N]$ by $k$ segments and on each of them it is linear on subsets of even and odd numbers. A map of positive integers $f: S \subseteq [1,N] \to [1,N]$ we call {\bf linear on $S$}, if there exist integers $a \in [1,N]$, $b \in [0,N]$, $f(i) = ai+b \pmod{N+1}$, where $i \in S$.

Using this construction we prove the next theorem:

\noindent {\bf The Width Theorem.} 
{\it For any $N>3$ there exists a solution which can be represented as a Queen function with width fewer or equal to $3$.}

The width of the function in the theorem cannot be reduced, since for $N=15$ our program checked that there are no solutions that can be represented as a Queen function with width $1$ or $2$. All other solutions by different authors have greater width than ours.

Next we introduce the concept of composition of solutions, formulate the criterion of being a composition of solutions a solution and prove this criterion for the generalization of composition. {\bf Generalized composition} of solutions $(A_1,A_2,\ldots, A_{|B|}) 
{\otimes} B$ is an arrangement which is defined by the following rule $$C(|B|(i-1) + j) = |B| (A_j(i)-1) + B(j),$$ where $|A_i|=|A_j| \ (1\leq i < j \leq |B|),$ $1\leq i \leq |A|$, $1 \leq j \leq |B|$.

\noindent {\bf Theorem.}
{\it Generalized composition $(A_1,A_2,\ldots, A_{|B|}) \otimes B$ of the solutions is a solution if and only if both of the following conditions hold:
\begin{enumerate}
\item $\{B(i) - i \pmod{|B|} \ | \ 1 \leqslant i \leqslant |B| \} = \mathbb{Z}_{|B|}.$
\item $\{B(i) + i \pmod{|B|} \ | \ 1 \leqslant i \leqslant |B| \} = \mathbb{Z}_{|B|}.$
\end{enumerate}}
Sufficiency in this theorem was proved by Polya in 1918 and we prove more complex part ~--- necessity.

If $A_1=A_2=\ldots=A_{|B|}$, we obtain {\bf The Composition Theorem}. 

Further we obtain two new corollaries from the Composition Theorem. 

And finally, we introduce definition of {\bf $Q$-irreducible} numbers $N$ for which none of the $N\times N$ solutions can be represented as a composition of smaller boards and definition of a {\bf fundamental set} of solutions which is a set of arrangements which generate all solutions to the $N$-Queens Problem by the rotation and reflection of the board. 

At the end, we formulate the conjecture that solves the $N$-Queens Problem and the $N$-Queens Completion for such $Q$-irreducible $N$ that $N-1$ is $Q$-irreducible in polynomial time:

\noindent{\bf Conjecture.}
{\it If numbers $N-1$ and $N$ are $Q$-irreducible, then there exists a set of fundamental solutions which can be represented as a Queen function with a width less or equal to $4$.}

Our program has checked this conjecture for $N$ up to $14$.

During the research new theoretical problems regarding prime numbers arise:
\begin{enumerate}
\item Are there infinitely many primes of form $2^k3^l-1$ (where $k,l \in \mathbb{N}$)? This is a generalization of Mersenne Primes Problem.
\item Are there infinitely many primes $p$ such that $2p+1$ is prime too?
\end{enumerate}

\noindent The paper is organized as follows. In the first section we introduce the new way of representation of queens' arrangements and prove the Width Theorem. In the second section we introduce composition of solutions, consider generalization of composition and prove the Composition Theorem. In the third section we consider connection between the width and composition and introduce a conjecture which solves the $N$-Queens Problem and $N$-Queens Completion in polynomial time.

\section{Width Theorem}

Any arrangement $A$ can be represented as a matrix or a permutation:

\begin{figure}[H]
\centering
\includegraphics[scale=0.125]{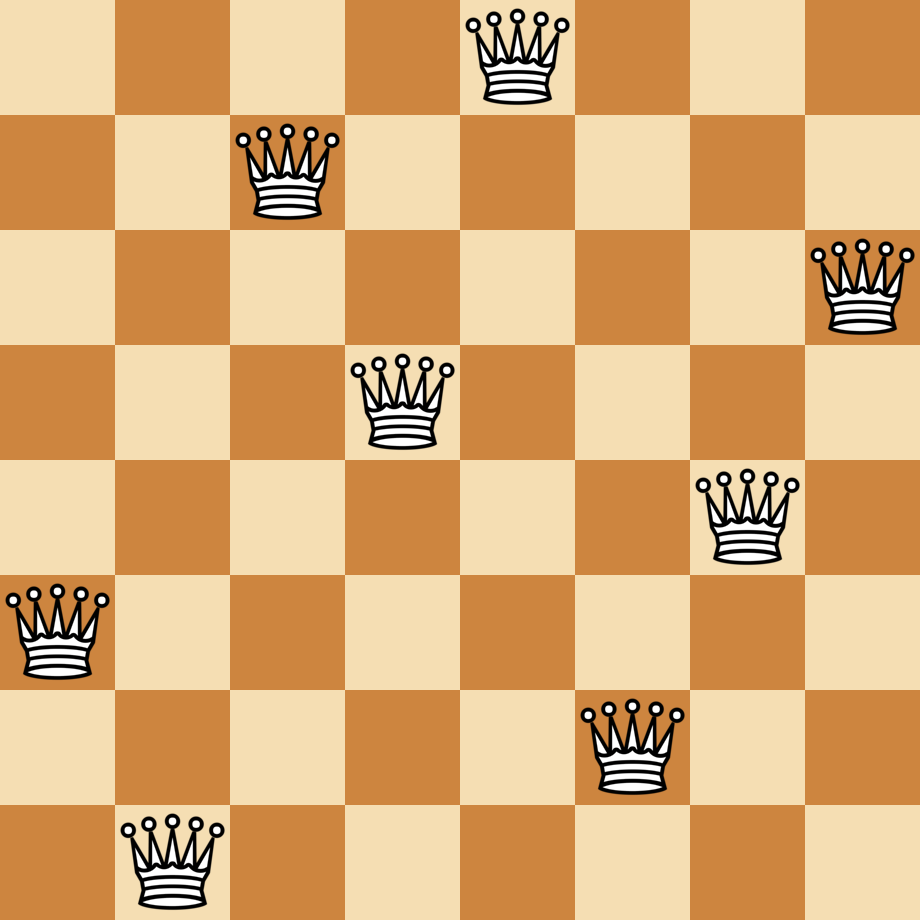}
     \caption{Arrangement $A$, $|A|=8$}
\end{figure}

$$A = \begin{pmatrix}
0 & 0 & 0 & 0 & 1 & 0 & 0 & 0\\
0 & 0 & 1 & 0 & 0 & 0 & 0 & 0\\
0 & 0 & 0 & 0 & 0 & 0 & 0 & 1\\
0 & 0 & 0 & 1 & 0 & 0 & 0 & 0\\
0 & 0 & 0 & 0 & 0 & 0 & 1 & 0\\
1 & 0 & 0 & 0 & 0 & 0 & 0 & 0\\
0 & 0 & 0 & 0 & 0 & 1 & 0 & 0\\
0 & 1 & 0 & 0 & 0 & 0 & 0 & 0\\
\end{pmatrix};$$

$$\begin{array}{|c|c|c|c|c|c|c|c|c|}
\hline
i & 1 & 2 & 3 & 4 & 5 & 6 & 7 & 8\\
\hline
A(i) & 3 & 1 & 7 & 5 & 8 & 2 & 4 & 6\\
\hline
\end{array}$$

\begin{Notation}
The size of $A$ is designated by $|A|$.
\end{Notation}

It is obvious that in any column and any line there must be exactly one queen. It means that the necessary condition of being an arrangement a solution is being a permutation. Now any ascending diagonal can be characterized by constant difference of queen's line and column. And any descending diagonal can be characterized by constant sum of queen's line and column.

That is why we can formulate the following well-known statement.

\begin{Statement}
Permutation $A$ is a solution to the $N$-Queens Problem if and only if $$\forall \ 1 \leqslant i < j \leqslant N \quad |i-j| \neq |A(i)-A(j)|.$$
\end{Statement}

Now we will introduce the new way of representing solutions.

\begin{Definition}
A map of positive integers $f: S \subseteq [1,N] \to [1,N]$ we call {\bf linear on $S$}, if there exist integers $a \in [1,N]$, $b \in [0,N]$, $f(i) = ai+b \pmod{N+1}$, where $i \in S$.
\end{Definition}

\begin{Definition}
{\bf A Queen function} is a map $[1,N] \to [1,N]$ which is defined on a partition of a segment $[1,N]$ by segments and on each of them it is linear on subsets of even and odd numbers.
\end{Definition}

\begin{Definition}
{\bf The width} of the Queen function is the quantity of segments of the partition.
\end{Definition}

It is easy to see that any $N\times N$ arrangement of queens can be represented as a Queen function with width $N$. So that, our aim is to find the upper bound of the estimate of width. 

\noindent {\bf The Width Theorem.} 
{\it For any $N>3$ there exists a solution which can be represented as a Queen function with width fewer or equal to $3$.}

\begin{proof}

To prove this theorem we will prove the following $5$ lemmas.

Nowadays existing solutions for $N=12k-4$ are difficult and their proofs are long and require considering lots of cases. The main lemma of the proof is following one that gives solution particularly for classic board $8\times 8$. This solution is the simplest of the known ones. 

\begin{Lemma}
Let $N=12k-4$. Then the following Queen function gives a solution
\begin{align*}
A(i) = 
\begin{cases}
2i \pmod{N+1}, & i \leq \frac{N}{2}\\
2i+2 \pmod{N+1}, & i > \frac{N}{2} \text{ and } i \text{ is odd}\\
2i-2 \pmod{N+1}, & i > \frac{N}{2} \text{ and } i \text{ is even}\\
\end{cases}.
\end{align*}
\end{Lemma}
\begin{proof}
For any $1 \leqslant i \leqslant N$ the value of $A(i)$ is not $0$ since $A(i)=2i \pmod{N+1}$ for $i\leqslant \frac{N}{2} \neq 0$. For odd $i>\frac{N}{2}$ suppose opposite. The value $A(i)= 2i+2 \pmod{N+1}=2i+2-N-1=0$. Then $2i+2=N+1$, but $i>\frac{N}{2}$ hence $2i+2>N+2$. For even $i$ suppose opposite. Then $A(i)=2i-2 \pmod{N+1}=2i-N-3=0$. Then $2i=N+3=12k-1$ but $12k-1$ is odd.

Now we will show that it is a permutation. For $i,j$ which are set by the same formula $A(i)\neq A(j)$. Let $i\leqslant \frac{N}{2}$, $j>\frac{N}{2}$.
\begin{enumerate}
\item $j$ is odd. Then $2i=2j+2-(N+1)$. Then $2(j-i+1)=N+1$, but $N+1$ is odd.
\item $j$ is even. Then $2i=2j-2-(N+1)$ or $2(j-i-1)=N+1$. But $N+1$ is odd.
\end{enumerate}

Now we will show that $|i-j|\neq|A(i)-A(j)|$. Suppose the opposite.
\begin{enumerate}
\item If for $i<j$ $A(i)$ and $A(j)$ are set by the formula $A(l)=2l \pmod {N+1}$, then $j-i=2j-2i$ or $1=2$. If they are set by formula $A(l)=2l+2 \pmod {N+1}$, then $j-i=2j+2-(N+1)-2i-2+(N+1)=2(j-i)$ or $1=2$. If they are set by formula $A(l)=2l-2 \pmod {N+1}$, then $j-i=2j-2-(N+1)-2i+2+(N+1)=2(j-i)$.
\item If $i\leqslant \frac{N}{2}$, $j>\frac{N}{2}$ and $j$ is odd, then $j-i=\pm(2i-(2j+2-(N+1)))=\pm(2i-2j+N-1)$.
\begin{enumerate}
\item $j-i=2i-2j+N-1$. Then $3(j-i)=N-1=12k-5$, but $12k-5$ is not divisible by $3$.
\item $j-i=2j-2i-N+1$. Then $j-i=N-1$. It is possible only for $j=N$, $i=1$, but then $j$ is even, which is another case.
\end{enumerate}
\item If $i\leqslant \frac{N}{2}$, $j>\frac{N}{2}$, $j$ is even, then $j-i=\pm(2i-(2j-2-N-1)=\pm(2i-2j+N+3)$.
\begin{enumerate}
\item $j-i=2i-2j+N+3$. Then $3(j-i)=N+3=12k-1$, but $12k-1$ is not divisible by $3$.
\item $j-i=2j-2i-N-3$. Then $j-i=N+3$, which is impossible since $j\leqslant n$.
\end{enumerate}
\end{enumerate}
Thus, $|i-j| \neq |A(i)-A(j)|$ and $A$ is a solution and for $N=12k-4$ there exists a solution with width $2$.
\end{proof}

\begin{Lemma}
Let $N=6k$ or $N=6k+4$. Then the following Queen function gives a solution 
\begin{align*}
A(i) = 2i \pmod{N+1}.
\end{align*}
\end{Lemma}
\begin{proof}
For any $1 \leqslant i \leqslant N$ the value of $A(i)$ is not $0$ since $N+1$ is not divisible by $2$. So all function values will be in range from $1$ to $N$. 

It will be a permutation: let $i\leqslant \frac{N}{2}$, $j>\frac{N}{2}$, suppose that $A(i)=A(j)$. Then $2i=2j-(N+1)$ or $2(i-j)=N+1$, but $N+1$ is odd.

Now let's show that this permutation is a solution. Suppose the opposite. 
Let $i>j$. Then $i-j=|(2i \pmod{N+1})-(2j \pmod{N+1})|$.
\begin{enumerate}
\item $i,j \leqslant \frac{N}{2}$. Then $i-j=2(i-j)$ or $1=2$.
\item $i,j>\frac{N}{2}$. Then $i-j=2i-(N+1)-(2j-(N+1))=2(i-j)$ or $1=2$.
\item $j\leqslant \frac{N}{2}$, $i>\frac{N}{2}$. Then
\begin{enumerate}
\item $i-j=2i-(N+1)-2j$. Then $i-j=N+1$ which is impossible.
\item $i-j=2j-(2i-(N+1))=2j-2i+N+1$ or $3(i-j)=N+1$, but $N$ is not divisible by $3$.
\end{enumerate}
\end{enumerate}
Thus, $|i-j| \neq |A(i)-A(j)|$ and $A$ is a solution and for $N=6k$ and $N=6k+4$ there exists a solution with width $1$.
\end{proof}

\begin{Lemma}
Let $N=6k+1$ or $N=6k+5$. Then the following Queen function gives a solution
\begin{align*}
A(i)=
\begin{cases}
2i \pmod{N+1}, & i < \frac{N}{2}\\
2i+1 \pmod{N+1}, & i> \frac{N}{2}\\
\end{cases}.
\end{align*}
\end{Lemma}
\begin{proof}
Obviously, it will be a permutation. Let $i<\frac{N}{2}$, $j>\frac{N}{2}$. Suppose that $A(i)=A(j)$ or $2i=2j+1-N-1$. Then $2i-2j=-N$ or $j-i=\frac{N}{2}$, contradiction ($\frac{N}{2}$ is not integer).

Now we will show that $|i-j|\neq|A(i)-A(j)|$. Suppose the opposite. 
\begin{enumerate}
\item $j<i<\frac{N}{2}$. Then $i-j=(2i \pmod{N+1})-(2j \pmod{N+1}) = 2i-2j$ or $1=2$.
\item $i>j>\frac{N}{2}$. Then $i-j=(2i+1 \pmod{N+1})-(2j+1 \pmod{N+1})=2i+1-(N+1)-(2j+1-(N+1))=2i-2j$ or $1=2$.
\item $j<\frac{N}{2}$, $i>\frac{N}{2}$. Then
\begin{enumerate}
\item $i-j=(2i+1 \pmod{N+1})-(2j \pmod{N+1})=2i+1-(N+1)-2j=2i-2j-N$. Then $i-j=N$, contradiction.
\item  $i-j=(2j \pmod{N+1})-(2i+1 \pmod{N+1})=2j-(2i+1-(N+1))=2j-2i+N$. Then $3(i-j)=N$, contradiction because $N$ is not divisible by $3$.
\end{enumerate}
\end{enumerate}
Thus, $|i-j| \neq |A(i)-A(j)|$ and $A$ is a solution and for $N=6k+1$ and $N=6k+5$ there exists a solution with width $2$.
\end{proof}

Note that $\{ 6k+2 \ | \ k \in \mathbb{N} \} = \{ 12k-4 \ | \ k \in \mathbb{N} \} \cup \{ 12k+2 \ | \ k \in \mathbb{N} \}.$ 

\begin{Lemma}
Let $N=12k+2$. Then the following Queen function gives a solution
\begin{align*}
A(i)=\begin{cases}
2i+4 \pmod {N+1}, & i<\frac{N}{2} \text{ and } i \text{ is odd or } i = N\\
2i \pmod {N+1}, & i<\frac{N}{2} \text{ and } i \text{ is even}\\
2i+2 \pmod {N+1}, & \frac{N}{2}\leq i < N.\\
\end{cases}
\end{align*}
\end{Lemma}
\begin{proof}
For any $1 \leqslant i \leqslant N$ the value of $A(i)$ is not $0$ since in the first part of the formula $i<=\frac{N}{2}$ (or $A(N)=2$) and odd $i$ by the formula $2i+4\leqslant N < N+1$ or $A(i) \neq 0$. In the second part $i\leqslant \frac{N}{2}-1$ we have $2i\leqslant N-2 < N+1$. In the third part $N+2 \leqslant 2i+2 < 2N+2$ or $2i+2 \pmod{n+1} \neq 0$.

Now we will show that it is a permutation. For $i,j$ which are set by the same formula $A(i)\neq A(j)$.. In other cases
\begin{enumerate}
\item $i<\frac{N}{2}$ and $\frac{N}{2} \leqslant j < N$. Suppose the opposite $A(i)=A(j)$.
\begin{enumerate}
\item $i$ is odd or $i=N$. Then $2i+4 \pmod{N+1} = 2j+2 \pmod{N+1}$ which is $2i+4=2j+2-(N+1)$ or $2j+2-(N+1)=2$. Then we get $2(j-i+1)=N+1$ or $2j=N+1$ which is impossible because $n+1=12k+3$ is odd.
\item $i$ is even. Then $2i \pmod{N+1} = 2j+2 \pmod{N+1}$ or $2i=2j+2-(N+1)$. And now we get $2(j-i+1)=N+1=12k+3$ contradiction.
\end{enumerate}
\item $i,j < \frac{N}{2}$ and $i$ is even, $j$ is odd. Then $2i \pmod{N+1} = 2j+4 \pmod{N+1}$ or $i-j=2$, but $i$ and $j$ have different parity.
\end{enumerate}

Now we will show that $|i-j|\neq|A(i)-A(j)|$. Suppose the opposite. If  for $i$ and $j$  $A$ is set by the same formula it is obvious. In other cases
\begin{enumerate}
\item $j<i<\frac{N}{2}$, $j$ is odd, $i$ is even. Then 
\begin{enumerate}
\item $i-j=(2i \pmod{N+1})-(2j + 4 \pmod{N+1}) = 2i-2j-4$ or $i-j=4$, but they have different parity.
\item $i-j=(2j+4 \pmod{N+1})-(2i \pmod{N+1}) = 2j+4-2i$ or $3(i-j)=4$ which is impossible.
\end{enumerate}
\item $j<\frac{N}{2}$, $j$ is odd and $\frac{N}{2}\leqslant i < N$. Then 
\begin{enumerate}
\item $i-j=(2i+2 \pmod{N+1})-(2j+4 \pmod{N+1})=2i+2-(N+1)-(2j+4)=2i-2j - N - 3$ or $i - j = N + 3$, impossible.
\item $i-j=(2j+4 \pmod{N+1})-(2i+2 \pmod{N+1})=2j+4-(2i+2-N-1)=2j-2i+N+3$ or $3(i-j) =N+3$ or $3(i-j)=12k+5$, impossible.
\end{enumerate}
\item $j<\frac{N}{2}$, $j$ is even and $\frac{N}{2}\leqslant i < N$. Then 
\begin{enumerate}
\item $i-j=(2i+2 \pmod{N+1})-(2j \pmod{N+1})=2i+2-(N+1)-2j=2i-2j-N+1$ or $i - j = N-1$, impossible.
\item $i-j=(2j \pmod{N+1})-(2i+2 \pmod{N+1})=2j-(2i+2-N-1)=2j-2i+N-1$ or $3(i-j)=N-1$ or $3(i-j)=12k+1$, impossible.
\end{enumerate}
\end{enumerate}
Thus, $|i-j| \neq |A(i)-A(j)|$ and $A$ is a solution and for $N=12k+2$ there exists a solution with width $3$.
\end{proof}

\begin{Lemma}
Let $N=6k+3$. Then the following Queen function gives a solution
\begin{align*}
A(i) = \begin{cases}
2i+2 \pmod{N+1}, & i < \frac{N-1}{2}\\
2i+4 \pmod{N+1}, & i = \frac{N-1}{2}\\
2i+5 \pmod{N+1}, & i > \frac{N-1}{2}.\\
\end{cases}
\end{align*}
\end{Lemma}
\begin{proof}
For any $1 \leqslant i \leqslant N$ the value of $A(i)$ is not $0$ since $i < \frac{N-1}{2}$, $A(i) = 2i \pmod{N+1} = 2i \neq 0$. For $i = \frac{N-1}{2}$ $A(i) = 2$. For $\frac{N-1}{2} < i < N-1$ $A(i) = 2i+5 \pmod{N+1} = 2i - N + 4 \neq 0$ and $A(N-1) = 1$, $A(n) = 3$. 


Now we will show that $|i-j|\neq|A(i)-A(j)|$. Suppose the opposite.
\begin{enumerate}
\item $i < \frac{N-1}{2}$, $j = \frac{N-1}{2}$. Then $j - i = \pm (2 - 2i - 2) = \pm 2i.$
\begin{enumerate}
\item $j-i=2i$. Then $2i=j=\frac{N-1}{2}=\frac{6k+2}{2}=3k+1$.
\item $j-i=-2i$. Then $i+j=0$.
\end{enumerate}
\item $i < \frac{N-1}{2}$, $\frac{N-1}{2} < j < N-1$. Then $j-i = \pm (2i+5-N-1-2i-2) = \pm (2(j-i)-N+2).$
\begin{enumerate}
\item $j-i = 2(j-i)-N+2$. Then $j-i=N-2$.
\item $j-i = -2(j-i)+N-2$. Then $3(j-i)=N-2=6k+1$.
\end{enumerate}
\item $i < \frac{N-1}{2}$, $j \geq N-1$. Then $j-i = \pm(2j+5-2N-2-2i-2)=\pm(2(j - i)-2N+1)$.
\begin{enumerate}
\item $j-i=2(j-i)-2N+1$. Then $j-i=2N-1$.
\item $j-i=-2(j-i)+2N-1$. Then $3(j-i)=2N-1=12k+5$.
\end{enumerate}
\item $i=\frac{N-1}{2}$, $\frac{N-1}{2} < j < N-1$. Then $j-i=\pm(2j+5-N-1-2)=\pm(2j-N+2)$.
\begin{enumerate}
\item $j-i=2j-N+2$. Then $j=N-2-i=\frac{2N-4-N+1}{2}=\frac{N-3}{2}<\frac{N-1}{2}.$
\item $j-i=N-2-2j$. Then $3j=N-2+i=\frac{2N-4+N-1}{2}=\frac{3N-5}{2}<\frac{3(N-1)}{2}$.
\end{enumerate}
\item $i=\frac{N-1}{2}$, $j \geqslant N-1$. Then $j-i=\pm(2j+5-2N-2-2)=\pm(2j-2N+1)$.
\begin{enumerate}
\item $j-i=2j-2N+1$. Then $j=2N-1-i=\frac{3N-1}{2}>N$.
\item $j-i=2N-2j-1$. Then $3j=2N-1+i=\frac{5N-3}{2}<3(N-1).$
\end{enumerate}
\item $\frac{N-1}{2}<i<N-1$, $j \geqslant N-1$. Then $j-i=\pm(2j+5-2N-2-2i-5+N+1)=\pm(2(j-i)-N-1).$
\begin{enumerate}
\item $j-i=2(j-i)-N-1.$ Then $j-i=N+1$.
\item $j-i=-2(j-i)+N+1$. Then $3(j-i)=n+1=6k+4.$
\end{enumerate}
\end{enumerate}
Thus, $|i-j| \neq |A(i)-A(j)|$ and $A$ is a solution and for $N=6k+3$ there exists a solution with width $3$.
\end{proof}
\noindent Thus, the width theorem is proved by building the examples of Queen functions for arbitrary values of $N$.
\end{proof}

\begin{Remark}
The width of the function in the theorem cannot be reduced, since for $N=15$ our program checked that there are no solutions that can be represented as a Queen function with width $1$ or $2$.
\end{Remark}

Also, all known solutions (1874 --- Glaisher, 1969 --- Hoffman, Loessi, Moore, 1991 --- Bernhardsson) have a greater width.

Here is the example of our solution and the easiest known solution for $N=20$.

\begin{figure}[H]
\begin{minipage}{.5\textwidth}
  \centering
  \includegraphics[width=.9\linewidth]{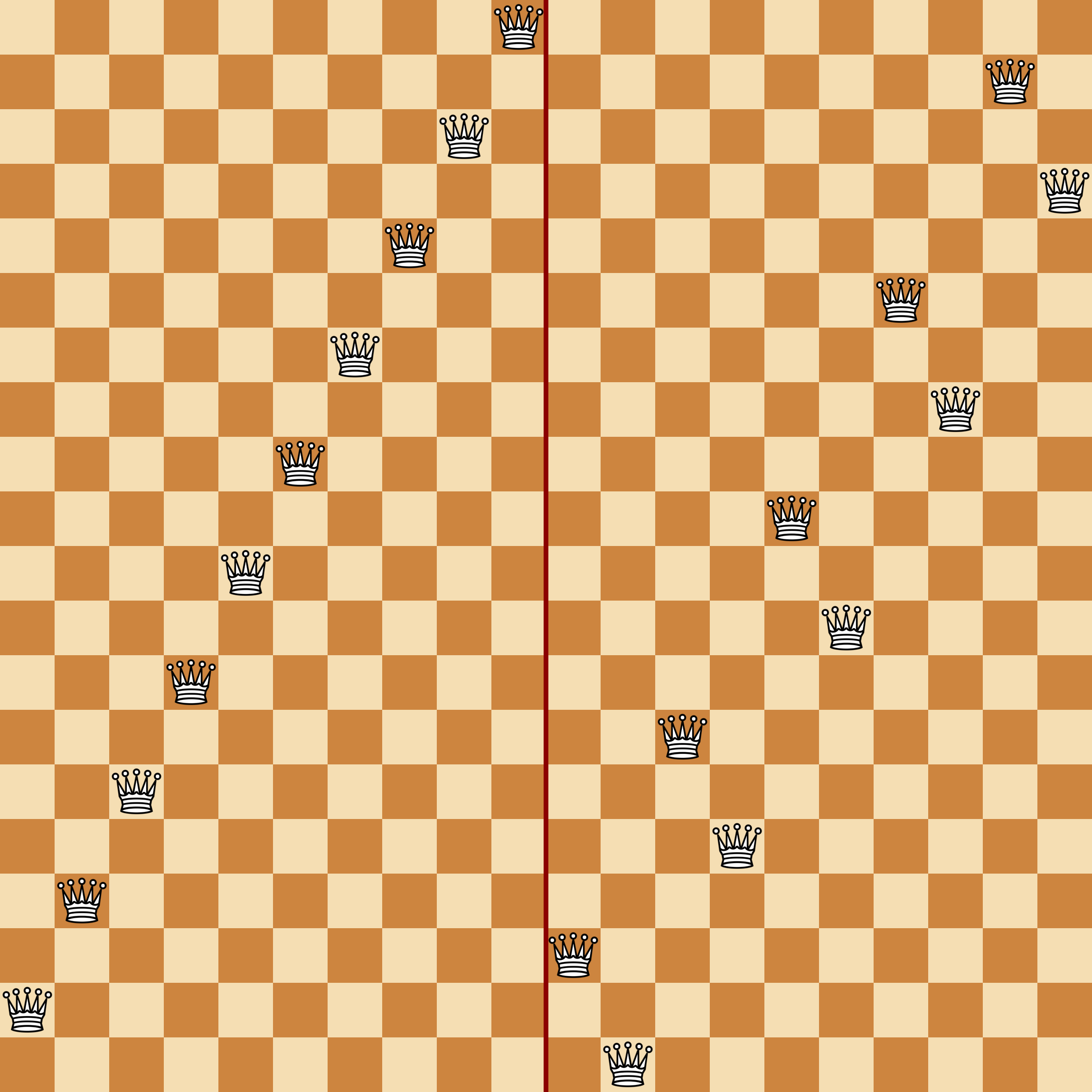}
  \captionof{figure}{\footnotesize Author's solution with width $2$}
\end{minipage}%
\begin{minipage}{.5\textwidth}
  \centering
  \includegraphics[width=.9\linewidth]{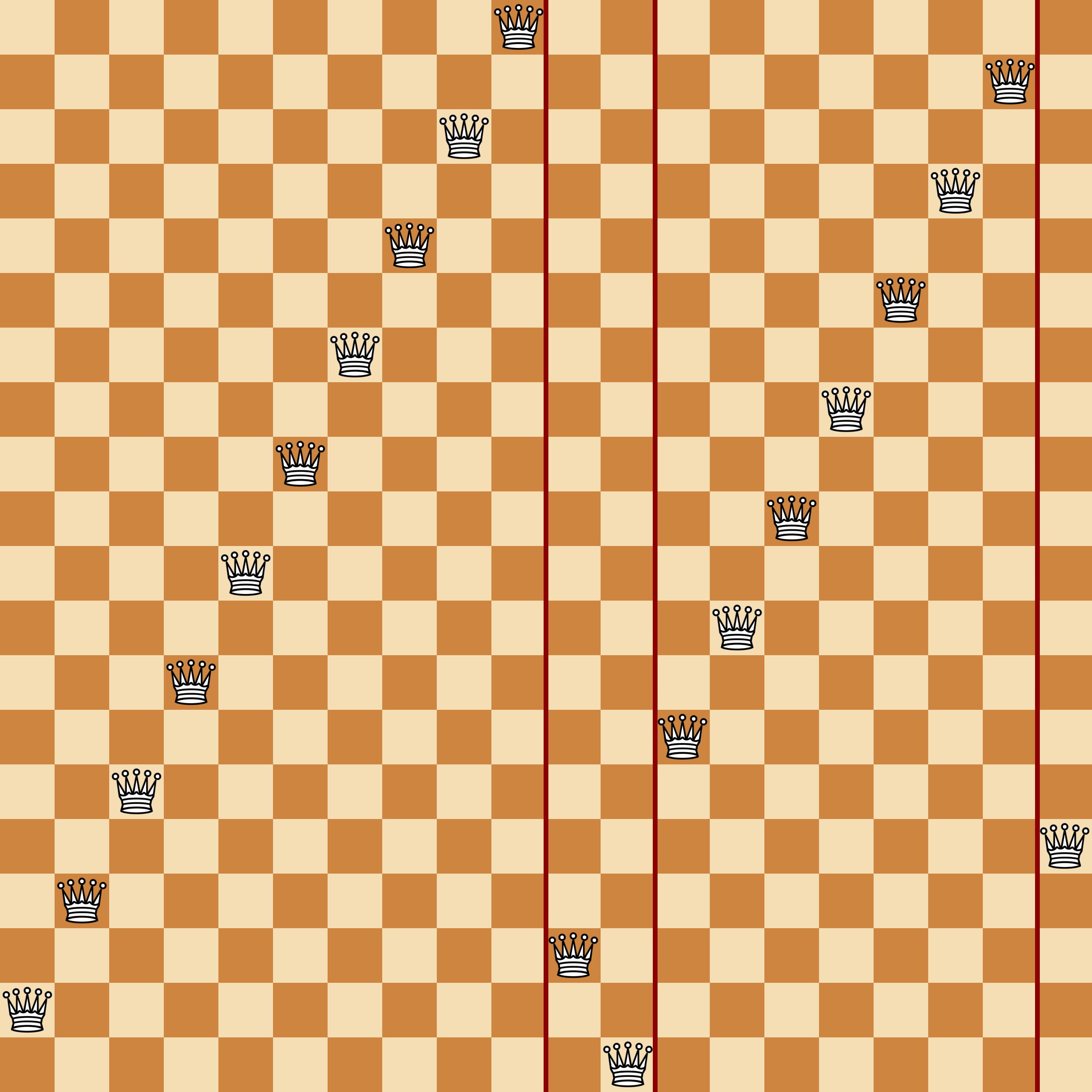}
  \captionof{figure}{\footnotesize Bernhardsson's solution with width $4$}
\end{minipage}
\end{figure}

\section{Composition Theorem}

\begin{Definition}
Given to arbitrary solutions $A$ and $B$. Then we can obtain new arrangement by insertion of one arrangement into queens’ positions of another arrangement. The obtained arrangement is called {\bf composition $A\otimes B$}.
\end{Definition}

For example consider two solutions $A$ and $B$:

\begin{figure}[H]
\begin{minipage}{.5\textwidth}
  \centering
  \includegraphics[width=.2\linewidth]{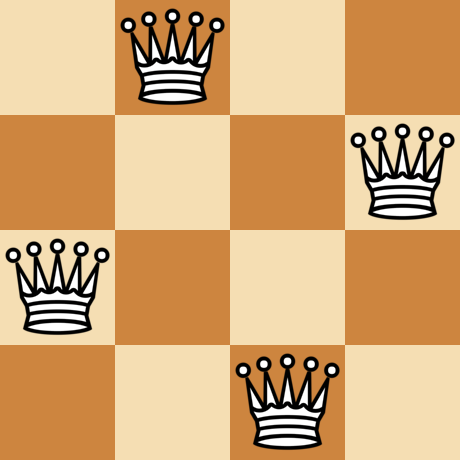}
\end{minipage}%
\begin{minipage}{.5\textwidth}
  \centering
  \includegraphics[width=.25\linewidth]{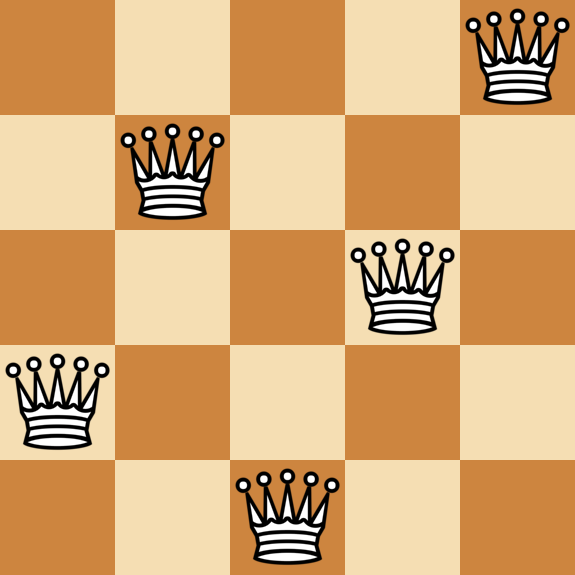}
\end{minipage}
\end{figure}

By we definition of composition we can obtain two new arrangements:

\begin{figure}[H]
\begin{minipage}{.5\textwidth}
	\centering
	\includegraphics[width=.9\linewidth]{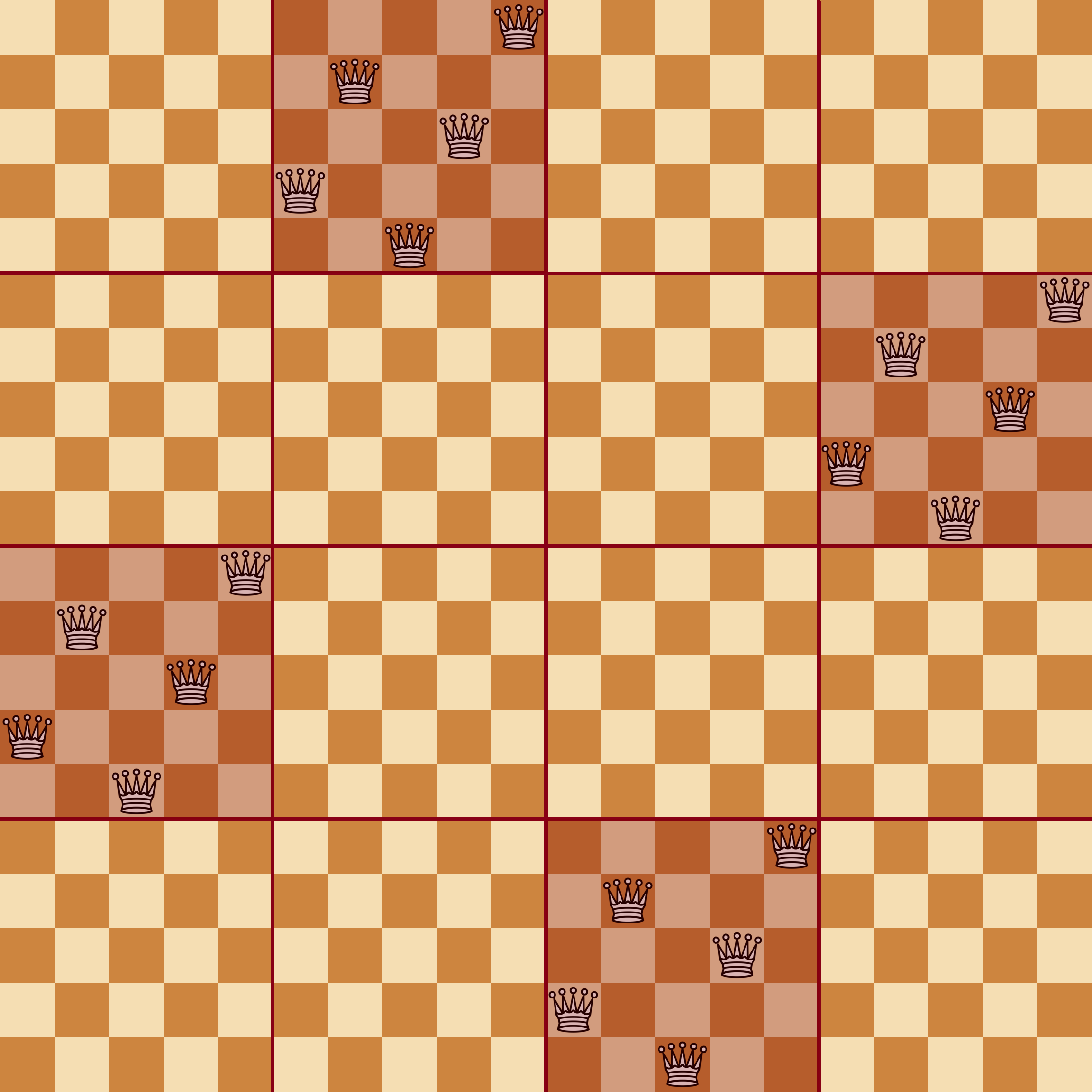}
	\captionof{figure}{$A\otimes B$ is a solution}
\end{minipage}%
\begin{minipage}{.5\textwidth}
	\centering
	\includegraphics[width=.9\linewidth]{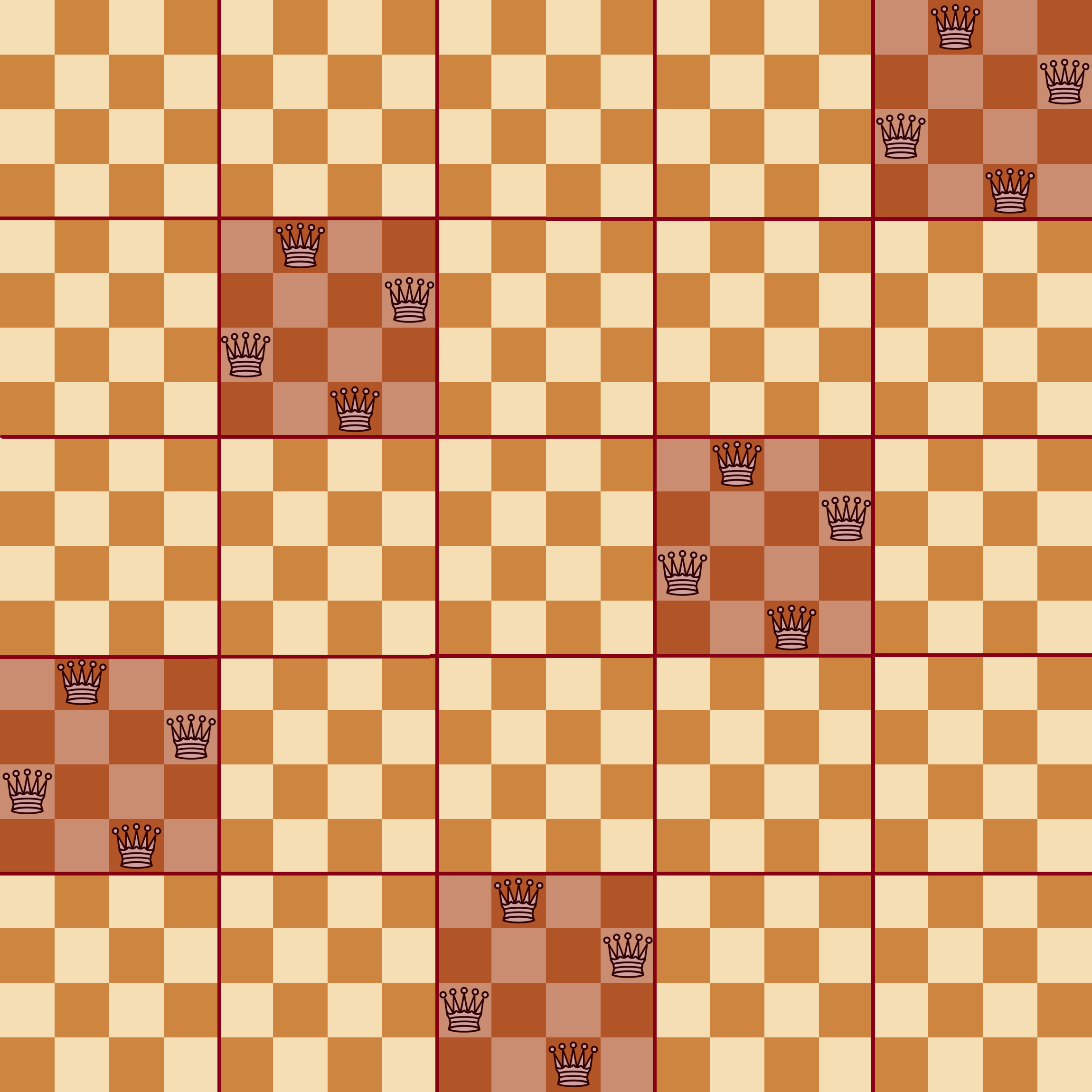}
	\captionof{figure}{$B\otimes A$ is not a solution}
\end{minipage}
\end{figure}

It is easy to see that composition of solutions is not always a solution. Thus we found necessary and sufficient condition of being a composition of solutions a solution:

{\bf The Composition Theorem.} 
{\it Composition $A \otimes B$ of the solutions $A$ and $B$ is a solution if and only if both of the following conditions hold:
\begin{enumerate}
\item $\{B(i) - i \pmod{|B|} \ | \ 1 \leqslant i \leqslant |B| \} = \mathbb{Z}_{|B|}.$
\item $\{B(i) + i \pmod{|B|} \ | \ 1 \leqslant i \leqslant |B| \} = \mathbb{Z}_{|B|}.$
\end{enumerate}}

To prove this theorem we will consider more general construction and prove the same criterion for it.

\begin{Definition}
{\bf Generalized composition} of solutions $(A_1,A_2,\ldots, A_{|B|}) 
{\otimes} B$ is an arrangement which is defined by the following rule $$C(|B|(i-1) + j) = |B| (A_j(i)-1) + B(j),$$ where $|A_i|=|A_j| \ (1\leq i < j \leq |B|),$ $1\leq i \leq |A|$, $1 \leq j \leq |B|$.
\end{Definition}

This construction is taken from \cite{Rivin} and firstly introduced by Polya in \cite{Polya}. 

Obviously, if arrangements $(A_1,A_2,\ldots, A_{|B|})$ are the same, then we get usual composition of solutions. 

\begin{Theorem}
Generalized composition $(A_1,A_2,\ldots, A_{|B|}) \otimes B$ of the solutions is a solution if and only if both of the following conditions hold:
\begin{enumerate}
\item $\{B(i) - i \pmod{|B|} \ | \ 1 \leqslant i \leqslant |B| \} = \mathbb{Z}_{|B|}.$
\item $\{B(i) + i \pmod{|B|} \ | \ 1 \leqslant i \leqslant |B| \} = \mathbb{Z}_{|B|}.$
\end{enumerate}
\end{Theorem}
\begin{proof}
Sufficiency was proved in 1918 by Polya in \cite{Polya}.

We will prove it again and also prove more complex part of this statement ~--- necessity.

Note that any integer $1 \leqslant m \leqslant |A||B|$ can be uniquely represented as $m = |B|(i-1)+j$, ($1 \leqslant i \leqslant |A|$, $1 \leqslant j \leqslant |B|$). Then for the chessboard $|A||B| \times |A||B|$ we will set the queens' arrangements by the formula
\begin{align*}
C(|B|(i - 1) + j) = |B|(A_j(i)-1) + B(j).
\end{align*}
Firstly, we will prove that this arrangement is a permutation. It is obvious since for different numbers their map images are either in different intervals or in one interval but with different values because $A_i, B$ are solutions.

Secondly, we will prove that this permutation is a solution. Let $1\leqslant r < s \leqslant |A||B|$. We have to prove $s-r \neq \pm(C(s) - C(r)) \neq \pm (C(|B|(i_1-1) + j_1) - C(|B|(i_2-1)+j_2))) \neq \pm (|B|(A_{j_1}(i_1)-1)+B(j_1)-|B|(A_{j_2}(i_2)-1)-B(j_2))\neq\pm(|B|(A_{j_1}(i_1)-A_{j_2}(i_2))+(B(j_1)-B(j_2)))$. We get that \begin{align}\label{gen}
|B|((i_1-i_2)\pm ((A_{j_2}(i_2)-A_{j_1}(i_1))) \neq \pm(B(j_1)-B(j_2))-(j_1-j_2).
\end{align}

First, we will prove sufficiency. Suppose the opposite: inequality \eqref{gen} doesn't hold. The conditions of the theorem are satisfied. Specifically, integers $B(i)-i$ are different modulo $|B|$ and integers $B(i)+i$ are different modulo $|B|$. Then the right-hand part of \eqref{gen} (now it is equality) is not divisible by $|B|$ and the required equality doesn't hold. 

Now we will prove the necessity. We consider cases of different sign before the expression:

\begin{enumerate}
\item $$|B|((A_{j_1}(i_1) + i_1) - (A_{j_2}(i_2) + i_2)) \neq (B(j_2) + j_2) - (B(j_1))+j_1).$$

Suppose the opposite. Such integers $j_1,j_2$, that $B(j_1) + j_1 \equiv B(j_2) + j_2 \pmod |B|$ exist. Then $|B|$ divides $B(j_2) + j_2 -(B(j_1) + j_1)$. Note that $-2|B|+3 \leqslant B(j_2) + j_2 -(B(j_2) + j_2) \leqslant 2|B|-3$. Then $B(j_2) + j_2 -(B(j_2) + j_2) = \pm |B|$. Then we will divide both parts of the inequality by $|B|$ and we get $(A_{j_1}(i_1) + i_1) - (A_{j_2}(i_2) + i_2) \neq \pm 1$. What is more $\forall 1 \leqslant p \leqslant |B| \quad 2 \leqslant A_p(i)+i \leqslant 2|A|$. Then we arrange in non-decreasing order $|A||B|$ integers $A_s(i)+i$ so, that $(A_{j_1}(i_1) + i_1) - (A_{j_2}(i_2) + i_2) \neq \pm 1$. There will be more than $|A|-1$ different integers among them. The following inequality $A_p(i)+i-(A_q(j)+j) \geqslant 2$ holds for them. Then we will add up all these inequalities and get that the difference between the largest and the least numbers is not fewer than $2|A|-2$. But the largest possible difference is $2|A|-3$. 

\item $$|B|((A_{j_2}(i_2) - i_2) - (A_{j_1}(i_1) - i_1)) \neq (B(j_1) - j_1) - (B(j_2))-j_2).$$

Suppose the opposite. Such integers $j_1,j_2$, that $B(j_1) - j_1 \equiv B(j_2) - j_2 \pmod |B|$ exist. Then $|B|$ divides $(B(j_1) - j_1) - (B(j_2))-j_2)$. Note that $-2|B|+3 \leqslant (B(j_1) - j_1) - (B(j_2))-j_2) \leqslant 2|B|-3$. Then $(B(j_1) - j_1) - (B(j_2))-j_2) = \pm |B|$. Then we will divide both parts of the inequality by $|B|$ and we get $(A_{j_2}(i_2) - i_2) - (A_{j_1}(i_1) - i_1) \neq \pm 1$. What is more $\forall 1 \leqslant p \leqslant |B| \quad 1-|A| \leqslant A_p(i)-i \leqslant |A| - 1$. Then we arrange in non-decreasing order $|A||B|$ integers $A_s(i)-i$ so, that $(A_{j_1}(i_1) - i_1) - (A_{j_2}(i_2) - i_2) \neq \pm 1$. here will be more than $|A|-1$ different integers among them. The following inequality $A_p(i)-i-(A_q(j)-j) \geqslant 2$ holds for them. Then we will add up all these inequalities and get that the difference between the largest and the least numbers is not fewer than $2|A|-2$. But the largest possible difference is $2|A|-3$. 
\end{enumerate}
Necessity is proved.
\end{proof}

\begin{proof}[Proof of The Composition Theorem]
This statement is a special case of {\bf the generalized composition theorem 2.3.} where $A_1=A_2=\ldots=A_{|B|}$.
\end{proof}

The next theorem is proved by Hedayat in \cite{Knut Vik design}:

\begin{Statement}[Hedayat's Lemma 2.2 in \cite{Knut Vik design}] \label{hedayat}
A permutation $B$ such that $\{B(i) - i \pmod{|B|} \ | \ 1 \leqslant i \leqslant |B| \} = \mathbb{Z}_{|B|}$ and $\{B(i) + i \pmod{|B|} \ | \ 1 \leqslant i \leqslant |B| \} = \mathbb{Z}_{|B|}$ exists if and only if $\gcd (|B|, 6) = 1$.
\end{Statement}

Using this statement we obtain the following two corollaries.
\begin{Corollary}
If $A\otimes B$ is a solution, then $\gcd(|B|,6)=1$. The opposite implication does not hold.
\end{Corollary}
\begin{proof}
If $A\otimes B$ is a solution, then $\{B(i) - i \pmod{|B|} \ | \ 1 \leqslant i \leqslant |B| \} = \mathbb{Z}_{|B|};$ $\{B(i) + i \pmod{|B|} \ | \ 1 \leqslant i \leqslant |B| \} = \mathbb{Z}_{|B|}$ by the composition theorem. And finally by the previous statement $\gcd(|B|,6)=1$.

But the opposite does not hold because for example for $|B|=7$ there exists a solution $$B = 
\begin{pmatrix}
1 & 2 & 3 & 4 & 5 & 6 & 7\\
4 & 7 & 5 & 2 & 6 & 1 & 3\\
\end{pmatrix},$$ which doesn't satisfy the required condition.
\end{proof}

\begin{Corollary}
If $A \otimes B$ is a solution, then $C \otimes B$ is a solution for any solution $C$.
\end{Corollary}
\begin{proof}
Based on the composition theorem, it is easy to see that criterion of being a composition of solutions $A\otimes B$ doesn't depend on the first arrangement $A$.
\end{proof}

\section{Connection between width and composition: conjecture}

It is easy to see that the width of the composition is quite large and it is impossible to find a constant bound on its width. According to this we introduce the next definition:

\begin{Definition}
Number $N$ is called {\bf $\bf Q$-irreducible}, if none of the $N\times N$ solutions can be represented as a composition of smaller boards.
\end{Definition}

\begin{Statement}
Number $N$ is $Q$-irreducible if and only if $N$ equals either $p$, or $2p$, or $3p$, or $2^k3^l$, where p is prime and $k,l \in \mathbb{N}_0$.
\end{Statement}
\begin{proof}
It is easy to see that if $N$ is equal to either $p$, or $2p$, or $3p$, or $2^k3^l$, then it is $Q$-irreducible.

Now we prove the necessity. Suppose the opposite. $N$ is neither $p$ or $2p$ or $3p$ or $2^k3^l$. 
Firstly, if $N=np$ (where $n > 3$), then there exists a solution which is a composition of arrangements $A$ and $B$, where $|A|=n$ and $|B|=p$.
Secondly, if $N=2^k3^l s$ (where $s \neq p$), then $s$ is not divisible by $2$ and $3$ and consequently $\gcd (s,6) = 1$. Thus, for such $N$ there exists a solution which is composition of smaller boards.
\end{proof}

Also operation of rotation of the board doesn't save the width of representation. Thus, we introduce the next definition:

\begin{Definition}
{\bf A fundamental set} of solutions is a set of arrangements which generate all solutions to the $N$-Queens Problem by the rotation and reflection of the board. 
\end{Definition}

Since generalized composition gives lots of solutions with large width, we consider $Q$-irreducible $N$. Also a queen can be arranged in the corner of the smaller board and $N \times N$ solution can be obtained. That is way $N-1$ must be $Q$-irreducible.

\begin{Conjecture}
If numbers $N-1$ and $N$ are $Q$-irreducible, then there exists a set of fundamental solutions which can be represented as a Queen function with a width less or equal to $4$.
\end{Conjecture}

This conjecture is checked by our program for $N$ up to $14$. Currently, the number of solutions is known only for values of $N \leqslant 27$.

For example, numbers $2018$ and $2019$ satisfy the condition of the conjecture (because $2017$ is a prime, $2018 = 2 \cdot 1009$ is a doubled prime, $2019 = 3 \cdot 673$ is a tripled prime).

If this conjecture is correct, then the $N$-Queens Problem and, consequently, the $N$-Queens Completion can be solved in polynomial time. Also then the number of solutions to the $N$-Queens Problem is bounded by the polynomial for such $Q$-irreducible $N$ that $N-1$ is also $Q$-irreducible. 

Based on the conjecture, new problem arise: are there infinitely many integers, satisfying the condition of the conjecture?

Let $N = 2^k 3^l$. Then $N-1$ is not divisible by $2$ and $3$. Thus $N-1$ must be prime and the question is the next one: are there infinitely many primes of form $2^k3^l-1$ (where $k,l \in \mathbb{N}_0$)? This is a generalization of Mersenne Primes Problem, because it is easy to see that for $l=0$ it is the Mersenne Primes Problem.

And the second problem is the next one: are there infinitely many primes $p$ such that $2p+1$ is prime too?

\section*{Acknowledgments}

The author would like to thank Dr. of Science (Habilitation) Stanislav Kublanovsky for problem statement and helpful discussion. Author is also grateful to Nikita Tepelin for assistance in writing programs.

\end{document}